\documentclass[10pt]{amsart}
\usepackage{mysymbols}
\usepackage{tensor}
\usepackage{braket,bm}

\newtheorem{thm}{Theorem}[section]
\newtheorem{prop}[thm]{Proposition}
\newtheorem{lem}[thm]{Lemma}
\newtheorem{cor}[thm]{Corollary}
\theoremstyle{definition}
\newtheorem{dfn}[thm]{Definition}
\theoremstyle{remark}

\newcommand{\ambientLichnerowicz}{\tilde{\Delta}_\mathrm{L}}
\newcommand{\ambientHodge}{\tilde{\Delta}_\mathrm{H}}
\newcommand{\divergence}{\delta}
\newcommand{\ceil}[1]{\lceil#1\rceil}
\DeclareMathOperator{\Image}{image}

\def\crn#1#2{{\vcenter{\vbox{
        \hbox{\kern#2pt \vrule width.#2pt height#1pt
           }
          \hrule height.#2pt}}}}
\def\intprod{\mathchoice\crn54\crn54\crn{3.75}3\crn{2.5}2}
\def\contraction{\mathbin{\intprod}}
\numberwithin{equation}{section}

\title[A GJMS construction for 2-tensors and the total $Q$-curvature]
{A GJMS construction for 2-tensors and\\
the second variation of the total $Q$-curvature}
\author{Yoshihiko Matsumoto}
\thanks{Partially supported by Grant-in-Aid for JSPS Fellows (22-6494).}
\address{Graduate School of Mathematical Sciences, The University of Tokyo}
\email{yoshim@ms.u-tokyo.ac.jp}

\begin{document}

\begin{abstract}
	We construct a series of conformally invariant differential operators acting on
	weighted trace-free symmetric 2-tensors by a method similar to Graham--Jenne--Mason--Sparling's.
	For compact conformal manifolds of dimension even and greater than or equal to four
	with vanishing ambient obstruction tensor,
	one of these operators is used to describe the second variation of the total $Q$-curvature.
	An explicit formula for conformally Einstein manifolds is given
	in terms of the Lichnerowicz Laplacian of an Einstein representative metric.
\end{abstract}

\maketitle

\section*{Introduction}

Let $(M,[g])$ be a conformal manifold of dimension $n\ge 3$.
The $k^\text{th}$ GJMS operator~\cite{Graham:1992fr} is a conformally invariant operator
acting on densities $\caE(-n/2+k)\to\caE(-n/2-k)$,
which is defined for all $k\in\bbZ_+$ if $n$ is odd and for integers
within the range $1\le k\le n/2$ if $n$ is even.
This operator has a universal expression in terms of any representative metric $g\in[g]$
with leading term the $k^\text{th}$ power of the Laplacian.
The idea for the construction is realizing densities as functions on the metric cone $\caG$
and computing the obstruction of its formal harmonic extension to the ambient space $(\tilde{\caG},\tilde{g})$,
where $\tilde{g}$ is an ambient metric of Fefferman--Graham~\cite{Fefferman:1985vy,Fefferman:2012vr}.
After the appearance of~\cite{Graham:1992fr}, other GJMS-like conformally invariant differential operators
have been constructed in, e.g., \cite{BransonGover:2005,GoverPeterson:2006}.

In this article, we establish another variant of the GJMS construction.
Our operators $P_k$ act on weighted trace-free symmetric (covariant) 2-tensors:
\begin{equation*}
	P_k\colon\caS_0\left(-\frac{n}{2}+2+k\right)\to\caS_0\left(-\frac{n}{2}+2-k\right).
\end{equation*}
Here, the values that $k$ takes are the same as in the density case,
$\caS_0$ is the space of trace-free symmetric 2-tensors on $M$, and $\caS_0(w)=\caS_0\otimes\caE(w)$.
The main tool of our construction is the Lichnerowicz Laplacian of the ambient metric $\tilde{g}$,
which is defined by
\begin{equation*}
	\ambientLichnerowicz:=\tilde{\Delta}+2\smash{\widetilde{\Ric}}^\circ-2\mathring{\tilde{R}},
\end{equation*}
where $\tilde{\Delta}=\tilde{\nabla}^*\tilde{\nabla}$ is the connection Laplacian and
$\smash{\widetilde{\Ric}}^\circ$, $\mathring{\tilde{R}}$ are the following tensorial actions of
the Ricci and Riemann curvature tensors of $\tilde{g}$:
\begin{align*}
	(\smash{\widetilde{\Ric}}^\circ\tilde{\sigma})(X,Y)
	&:=\frac{1}{2}\left(\braket{\widetilde{\Ric}(X,\cdot),\tilde{\sigma}(Y,\cdot)}_{\tilde{g}}
	+\braket{\widetilde{\Ric}(Y,\cdot),\tilde{\sigma}(X,\cdot)}_{\tilde{g}}\right),\\
	(\mathring{\tilde{R}}\tilde{\sigma})(X,Y)
	&:=\braket{\tilde{R}(X,\cdot,Y,\cdot),\tilde{\sigma}}_{\tilde{g}}.
\end{align*}

Our intention to study the GJMS construction for 2-tensors is because of its relation to
the second variation of the total $Q$-curvature, i.e.,
the integral of Branson's $Q$-curvature~\cite{Branson:1995db}.
Recall that, for 4-dimensional compact conformal manifold $(M,[g])$ of positive-definite signature,
the Chern--Gauss--Bonnet formula for the total $Q$-curvature $\overline{Q}$ is
\begin{equation*}
	\overline{Q}=8\pi^2\chi(M)-\frac{1}{4}\int_M\abs{W}_g^2\,dV_g,
\end{equation*}
where $\chi(M)$ is the Euler characteristic, $W$ is the Weyl tensor, and $g\in[g]$ is any representative.
One can deduce from this that $\overline{Q}\le 8\pi^2\chi(M)$ and
the equality holds if and only if $(M,[g])$ is conformally flat.
It turns out that there is a partial generalization of this fact to the higher dimensions.
Recall that a conformal metric $[g]$ is identified with a weighted 2-tensor $\bm{g}\in\caS(2)$.
Let $\caK_{[g]}$ be the conformal Killing operator.
Then we have the following theorem, which is due to M\o{}ller--\O{}rsted~\cite{MollerOrsted}.
\begin{thm}
	\label{thm:Sphere}
	Let $S^n$ be the sphere of even dimension $n\ge 4$.
	Then, for any smooth 1-parameter family $\bm{g}_t$ of conformal metrics on $S^n$
	such that $\bm{g}_0=\bm{g}_\mathrm{std}$ and $\dot{\bm{g}}_t|_{t=0}\not\in\Image\caK_{[g_\mathrm{std}]}$,
	the total $Q$-curvature $\overline{Q}_t$ attains a local maximum at $t=0$.
\end{thm}
Our main theorem contains Theorem \ref{thm:Sphere} as a special case.
Consider the following decomposition of $\caS_0(2)$, which is valid for any compact positive-definite
conformal manifold $(M,[g])$ and a representative $g\in[g]$ (see~\cite[12.21]{Besse:1987vf}):
\begin{equation}
	\label{eq:TTDecomposition}
	\caS_0(2)=\Image\caK_{[g]}\oplus\caS^g_\mathrm{TT}(2).
\end{equation}
Here $\caS^g_\mathrm{TT}(w)$ is the space of
TT-tensors (trace-free and divergence-free tensors) with respect to $g$.
This is an orthogonal decomposition with respect to the $L^2$-inner product,
and if $g$ is Einstein, the Lichnerowicz Laplacian $\Delta_\mathrm{L}$ of $g$ respects this decomposition.
\begin{thm}
	\label{thm:QCurvLocalMaximum}
	Let $(M,[g])$ be a compact conformally Einstein manifold of positive-definite signature
	with even dimension $n\ge 4$,
	and $g$ an Einstein representative with Schouten tensor $\tensor{P}{_i_j}=\lambda\tensor{g}{_i_j}$.
	Then, for any smooth 1-parameter family $\bm{g}_t$ of conformal metrics such that $\bm{g}_0=\bm{g}$,
	the second derivative of the total $Q$-curvature at $t=0$ is
	\begin{equation}
		\label{eq:SecondDerivativeOfTotalQForConfEinstein}
		\frac{d^2}{dt^2}\overline{Q}_t
		=-\frac{1}{4}\int_M
		\Braket{\prod_{m=0}^{n/2-1}(\Delta_\mathrm{L}-4(n-1)\lambda+4m(n-2m-1)\lambda)\varphi^g_\mathrm{TT},
		\varphi^g_\mathrm{TT}}_{\bm{g}},
	\end{equation}
	where $\varphi^g_\mathrm{TT}$ is the $\caS^g_\mathrm{TT}(2)$-component of $\varphi=\dot{\bm{g}}_t|_{t=0}$
	with respect to \eqref{eq:TTDecomposition}.
	In particular, suppose there is an Einstein representative $g$ with $\lambda\ge 0$
	such that the smallest eigenvalue $\alpha$ of $\Delta_\mathrm{L}|_{\caS_\mathrm{TT}^g(2)}$ satisfies
	\begin{equation}
		\label{eq:EigenvalueCondition}
		\alpha>4(n-1)\lambda.
	\end{equation}
	Then, for any $\bm{g}_t$ for which $\varphi\not\in\Image\caK_{[g]}$,
	the total $Q$-curvature attains a local maximum at $t=0$.
\end{thm}

For $(S^n,g_\mathrm{std})$, $\lambda=1/2$ and $\Delta_\mathrm{L}=\Delta+2n$.
Therefore the assumption for the latter half of Theorem \ref{thm:QCurvLocalMaximum} is satisfied,
and hence Theorem \ref{thm:Sphere} follows.

Some ideas for the proof of Theorem \ref{thm:QCurvLocalMaximum} are in order.
Let $(M,[g])$ be a compact conformal manifold of even dimension $n\ge 4$
(here we may allow arbitrary signature).
If we are given a smooth family $\bm{g}_t$ of conformal metrics on $M$ such that $\bm{g}_0=\bm{g}$,
then the derivative $\varphi_t=\dot{\bm{g}}_t\in\caS(2)$ is trace-free with respect to $\bm{g}_t$.
As shown by Graham--Hirachi~\cite{Graham:2005cx},
the derivative of $\overline{Q}_t$ is given by
\begin{equation*}
	\frac{d}{dt}\overline{Q}_t=
	(-1)^{n/2}\frac{n-2}{2}\int_M\braket{\caO_t,\varphi_t}_{\bm{g}_t},
\end{equation*}
where $\caO_t$ is the Fefferman--Graham ambient obstruction tensor of
$\bm{g}_t$~\cite{Fefferman:1985vy,Fefferman:2012vr}.
In particular, if $(M,[g])$ has vanishing obstruction tensor, which is the case if $(M,[g])$ is
conformally Einstein for instance, then $\overline{Q}_t$ stabilizes at $t=0$.
In this case the second derivative of $\overline{Q}_t$ at $t=0$ is of interest. It is given by
\begin{equation}
	\label{eq:SecondDerivativeOfTotalQAndDifferentialObstruction}
	\left.\frac{d^2}{dt^2}\overline{Q}_t\right|_{t=0}
	=(-1)^{n/2}\frac{n-2}{2}\int_M\braket{\caO'_{\bm{g}}\varphi,\varphi}_{\bm{g}},
\end{equation}
where $\caO'_{\bm{g}}\colon\caS_0(2)\to\caS_0(2-n)$ is the linearization at $\bm{g}$ of
the obstruction tensor operator
($\caO'_{\bm{g}}\varphi$ is trace-free because $\bm{g}$ is obstruction-flat).
This shows that it suffices to compute $\caO'_{\bm{g}}$
to derive the second variational formula of the total $Q$-curvature.
The construction of our operators $P_k$ leads to the fact that $P=P_{n/2}$ is equal to $\caO'_{\bm{g}}$
up to a constant factor for obstruction-flat manifolds.
(For $n=4$ and $6$, since an explicit formula of the obstruction tensor is known,
one can directly compute its linearization.
In higher dimensions our result is really new, because there is no such concrete formula for $\mathcal{O}$.)
Thus our GJMS construction adds new knowledge of $\caO'_{\bm{g}}$, which is previously
studied in \cite{Branson:2005,BransonGover:2007,BransonGover:2008}.

If we specialize to the case of conformally Einstein manifolds,
explicit computation is possible thanks to a well-known associated ambient metric.
We will derive a formula of $P_k$ restricted to $\caS^g_\mathrm{TT}(-n/2+2+k)$
with respect to an Einstein representative $g$ with Schouten tensor $\tensor{P}{_i_j}=\lambda\tensor{g}{_i_j}$:
\begin{equation}
	\label{eq:Intro:CharacteristicGJMSOperatorOnTTTensor}
	\begin{split}
		&P_k|_{\caS^g_\mathrm{TT}(-n/2+2+k)}\\
		&=\prod_{m=0}^{k-1}
		\left(\Delta_\mathrm{L}-4(n-1)\lambda-2\left(-\frac{n}{2}+k-2m\right)\left(\frac{n}{2}+k-2m-1\right)\lambda
		\right).
	\end{split}
\end{equation}
Then Theorem \ref{thm:QCurvLocalMaximum} is an immediate consequence.

This article is organized as follows.
Preliminaries about ambient metrics and some preparatory lemmas are included in Section \ref{sec:Preliminaries}.
In Section \ref{sec:GJMS}, our operators $P_k$ are constructed.
One of the characterizations of $P_k$ is that it gives the obstruction to dilation-annihilating
TT-harmonic extension of $\varphi\in\caS_0(-n/2+2+k)$ with respect to
the ambient Lichnerowicz Laplacian $\ambientLichnerowicz$.
In Section \ref{sec:Variations}, we first show that the variation of the normal-form ambient metric
modified by adding a certain tensor in the image of the Killing operator of $\tilde{g}$
is a best possible approximate solution to the harmonic extension problem mentioned above.
Using this fact, we prove that the trace-free part of $\caO'_{\bm{g}}$ equals to $P$ in general.
In Section \ref{sec:ConfEinstein}, we work on conformally Einstein manifolds
and prove Theorem \ref{thm:QCurvLocalMaximum}.

In this article, ``conformal manifolds'' are of arbitrary signature unless otherwise stated.
Index notation is used throughout.
On ambient spaces we use $I$, $J$, $K$, $\dotsc$ as indices, while on the original manifolds
$i$, $j$, $k$, $\dots$ are used.

I wish to thank Kengo Hirachi for a suggestion to take a variational approach to the $Q$-curvature
and for insightful advice, and Robin Graham for discussion on our
formula \eqref{eq:SecondDerivativeOfTotalQForConfEinstein} and letting me know the work \cite{MollerOrsted}.
I also thank Bent \O{}rsted for related discussions.
Moreover, Colin Guillarmou informed me that he has given another proof of
\eqref{eq:SecondDerivativeOfTotalQForConfEinstein} with Sergiu Moroianu and Jean-Marc Schlenker
in a recent work \cite{GuillarmouMoroianuSchlenker}.

\section{Preliminaries}
\label{sec:Preliminaries}

Let $(M,[g])$ be a conformal manifold of dimension $n$ of signature $(p,q)$ with metric cone $\caG$.
With a fixed representative metric $g\in[g]$, $\caG$ is trivialized as
\begin{equation*}
	\caG\cong \bbR_+\times M,\qquad
	t^2g_x\mapsto (t,x).
\end{equation*}
Let $\tilde{\caG}$ be the associated ambient space:
\begin{equation*}
	\tilde{\caG}:=\caG\times\bbR\cong \bbR_+\times M\times\bbR=\set{(t,x,\rho)}.
\end{equation*}
In our index notation, if $\tilde{\caG}$ is trivialized as above, we use the indices $0$ and $\infty$
for the $t$- and $\rho$-components, respectively.

The space $\caG$ carries a natural $\bbR_+$-bundle structure.
The dilation $\delta_s$, $s\in\bbR^\times$, is by definition the action of $s^2\in\bbR_+$,
and the infinitesimal dilation field is denoted by $T$.
The spaces of the densities, weighted 1-forms, and weighted covariant symmetric 2-tensors
(all of weight $w$) are denoted by $\caE(w)$, $\caT(w)$, and $\caS(w)$.
By the metric cone $\caG$, these spaces are realized as follows:
\begin{equation}
	\label{eq:DensitiesAsAmbientTensors}
	\begin{split}
		\caE(w)&=\set{f\in C^\infty(\caG,\bbR)|Tf=wf},\\
		\caT(w)&=\set{\tau\in C^\infty(\caG,T^*\caG)|T\contraction\tau=0,\quad\caL_T\tau=w\tau},\\
		\caS(w)&=\set{\sigma\in C^\infty(\caG,\Sym^2T^*\caG)|T\contraction\sigma=0,\quad\caL_T\sigma=w\sigma}.
	\end{split}
\end{equation}
The $\bbR_+$-action extends to $\tilde{\caG}=\caG\times\bbR$ and so does $T$.
In terms of the extended $T$, we define as follows:
\begin{align*}
	\tilde{\caE}(w)&:=\set{\tilde{f}\in C^\infty(\tilde{\caG},\bbR)|T\tilde{f}=w\tilde{f}},\\
	\tilde{\caT}(w)
	&:=\set{\tilde{\tau}\in C^\infty(\tilde{\caG},T^*\tilde{\caG})|\caL_T\tilde{\tau}=w\tilde{\tau}},\\
	\tilde{\caS}(w)
	&:=\set{\tilde{\sigma}\in C^\infty(\tilde{\caG},\Sym^2T^*\tilde{\caG})|\caL_T\tilde{\sigma}=w\tilde{\sigma}}.
\end{align*}
When $\tilde{\sigma}\in\tilde{\caS}(w)$ satisfies $(T\contraction\tilde{\sigma})|_{T\caG}=0$,
then $\tilde{\sigma}|_{T\caG}$ makes sense as a section in $\caS(w)$ via the identification
\eqref{eq:DensitiesAsAmbientTensors}.
We use the notation $\tilde{\sigma}|_{TM}$ to express this weighted tensor.

Let $\tilde{g}$ be a pre-ambient metric.
This means that $\tilde{g}\in\tilde{\caS}(2)$ is a homogeneous pseudo-Riemannian metric of signature $(p+1,q+1)$
defined on a dilation-invariant open neighborhood of $\caG$ in $\tilde{\caG}$
such that its pullback to $\caG$ is equal to $\bm{g}\in\caS(2)$.
In the sequel we only work asymptotically near $\caG$, so we may assume that all pre-ambient metrics are
defined on the whole $\tilde{\caG}$.
We next introduce the \emph{straightness} condition:
\begin{equation}
	\label{eq:Straightness}
	\tilde{\nabla}T=\id.
\end{equation}
If this is the case, the differential of the canonical defining function
$r=\abs{T}_{\tilde{g}}^2$ of $\caG$ is
\begin{equation}
	\label{eq:Straightness2}
	dr=2T\contraction\tilde{g}.
\end{equation}
Recall that it follows immediately from \eqref{eq:Straightness} that
\begin{equation}
	\label{eq:TAmbientCurvature}
	\tensor{T}{^I}\tensor{\tilde{R}}{_I_J_K_L}=0,\quad
	\text{and hence}\quad\tensor{T}{^I}\tensor{\widetilde{\Ric}}{_I_J}=0.
\end{equation}
The Fefferman--Graham Theorem states that there is a straight pre-ambient metric $\tilde{g}$ with
\begin{equation*}
	\widetilde{\Ric}=
	\begin{cases}
		O(r^\infty)&\text{if $n$ is odd},\\
		O(r^{n/2-1})&\text{if $n$ is even}.
	\end{cases}
\end{equation*}
In this article, such a metric $\tilde{g}$ is called an \emph{ambient metric}.
When $n$ is odd, ambient metrics are unique modulo $O(r^\infty)$ and the action of
dilation-invariant diffeomorphisms on $\tilde{\caG}$ leaving points on $\caG$ fixed
(such diffeomorphisms are called \emph{ambient-equivalence maps} in the sequel).
If is $n$ even, the situation is subtle.
For a 1-form $\tilde{\tau}\in\tilde{\caT}(w)$, we define
\begin{equation*}
	\begin{split}
		\tilde{\tau}=O^-(r^m)&\Longleftrightarrow
		\tilde{\tau}=O(r^{m-1})\quad\text{and}\quad\text{$(r^{1-m}\tilde{\tau})|_{T\caG}$ vanishes}\\
		&\Longleftrightarrow
		\tilde{\tau}=O(r^m)\mod r^{m-1}T\contraction\tilde{g}.
	\end{split}
\end{equation*}
We say that $\tilde{\sigma}\in\tilde{\caS}(w)$ is $O^+(r^m)$ if
\begin{enumerate}
	\item $\tilde{\sigma}=O(r^m)$;
	\item $T\contraction\tilde{\sigma}=O^-(r^{m+1})$ and hence
		$(r^{-m}\tilde{\sigma})|_{TM}$ makes sense; and
	\item $(r^{-m}\tilde{\sigma})|_{TM}\in\caS(w-2m)$ is trace-free with respect to $\bm{g}$.
\end{enumerate}
Then, ambient metrics are unique modulo $O^+(r^{n/2})$ and the action of ambient-equivalence maps.
By \cite[Equation (3.13)]{Fefferman:2012vr}, the condition $\widetilde{\Ric}=O(r^{n/2-1})$ for
ambient metrics actually forces
\begin{equation*}
	\widetilde{\Ric}=O^+(r^{n/2-1}).
\end{equation*}

Let $g\in[g]$ and consider the induced trivialization $\tilde{\caG}\cong\bbR_+\times M\times\bbR$.
If a straight pre-ambient metric $\tilde{g}$ is near $\caG$ of the form
\begin{equation}
	\label{eq:NormalForm}
	\tilde{g}=2\rho\,dt^2+2\rho\,dt\,d\rho+t^2g_\rho,
\end{equation}
where $g_\rho$ is a 1-parameter family of metrics on $M$ with $g_0=g$,
then $\tilde{g}$ is said to be in \emph{normal form relative to $g$}.
For any straight pre-ambient metric $\tilde{g}$ and a choice of $g\in[g]$,
it is known~\cite[Proposition 2.8]{Fefferman:2012vr} that
there exists an ambient-equivalence map $\Phi$ such that $\Phi^*\tilde{g}$ is in normal form relative to $g$.

\begin{lem}
	\label{lem:ActionOfNablaT}
	Let $\tilde{g}$ be a straight pre-ambient metric. For $\tilde{\tau}\in\tilde{\caT}(w)$ and
	$\tilde{\sigma}\in\tilde{\caS}(w)$,
	\begin{equation*}
		\tilde{\nabla}_T\tilde{\tau}=(w-1)\tilde{\tau},\qquad
		\tilde{\nabla}_T\tilde{\sigma}=(w-2)\tilde{\sigma}.
	\end{equation*}
\end{lem}

\begin{proof}
	Let $\tilde{\xi}\in\frX(\tilde{\caG})$. Then, since the Levi-Civita connection is torsion-free,
	\begin{equation*}
		\begin{split}
			(\tilde{\nabla}_T\tilde{\tau})(\tilde{\xi})
			&=T(\tilde{\tau}(\tilde{\xi}))-\tilde{\tau}(\tilde{\nabla}_T\tilde{\xi})
			=T(\tilde{\tau}(\tilde{\xi}))-\tilde{\tau}([T,\tilde{\xi}]+\tilde{\nabla}_{\tilde{\xi}}T)\\
			&=T(\tilde{\tau}(\tilde{\xi}))
			-\tilde{\tau}(\caL_T\tilde{\xi})-\tilde{\tau}(\tilde{\nabla}_{\tilde{\xi}}T)
			=(\caL_T\tilde{\tau})(\tilde{\xi})-\tilde{\tau}(\tilde{\xi})
			=(w-1)\tilde{\tau}(\tilde{\xi}).
		\end{split}
	\end{equation*}
	The second equality is proved similarly.
\end{proof}

Now let $\tilde{g}$ be a fixed ambient metric.
Let $\tilde{\caS}_0(w)$ be the subspace of formally trace-free tensors of $\tilde{\caS}(w)$,
and $\tilde{\caS}_\mathrm{TT}(w)$ the subspace of formally TT-tensors.
Moreover, we define as follows:
\begin{gather*}
	\tilde{\caS}^X(w):=
	\set{\tilde{\sigma}\in\tilde{\caS}(w)|T\contraction\tilde{\sigma}=O(r^\infty)},\\
	\tilde{\caS}_0^X(w):=\tilde{\caS}_0(w)\cap\tilde{\caS}^X(w),\qquad
	\tilde{\caS}_\mathrm{TT}^X(w):=\tilde{\caS}_\mathrm{TT}(w)\cap\tilde{\caS}^X(w).
\end{gather*}
If $n$ is odd, these spaces are invariant under $O(r^\infty)$-modifications of $\tilde{g}$.
If $n$ is even, we need some technically-defined tensor spaces.
For $2-n\le w\le 2$, we set
\begin{equation*}
	\tilde{\caS}_\mathrm{aTT}(w)
	:=\set{\tilde{\sigma}\in\tilde{\caS}(w)|
	\tr_{\tilde{g}}\tilde{\sigma}=O(r^{\ceil{\frac{n-2+w}{2}}}),\quad
	\divergence_{\tilde{g}}\tilde{\sigma}=O^-(r^{\ceil{\frac{n-2+w}{2}}})}
\end{equation*}
(``aTT'' is for ``approximately TT'') and
\begin{equation*}
	\tilde{\caS}_\mathrm{aTT}^X(w)
	:=\set{\tilde{\sigma}\in\tilde{\caS}_\mathrm{aTT}(w)|T\contraction\tilde{\sigma}
	=O^-(r^{\ceil{\frac{n-2+w}{2}}+1})},
\end{equation*}
where
$\divergence_{\tilde{g}}$ is the divergence operator
$\tensor{(\divergence_{\tilde{g}}\tilde{\sigma})}{_I}=-\tensor{\tilde{\nabla}}{^J}\tensor{\tilde{\sigma}}{_I_J}$,
and $\ceil{x}$ is the smallest integer not less than $x$.
Then $\tilde{\caS}_\mathrm{aTT}^X(w)$ does not depend on the $O^+(r^{n/2})$-ambiguity of $\tilde{g}$.
To check this, let $\tilde{g}'=\tilde{g}+A$ be another ambient metric with $A=O^+(r^{n/2})$.
Then $T\contraction A=O^-(r^{n/2+1})$.
Since $\tr_{\tilde{g}'}\tilde{\sigma}=\tr_{\tilde{g}}\tilde{\sigma}+O(r^{n/2})$ for any $\tilde{\sigma}$,
the trace condition is not affected.
The Christoffel symbol of $\tilde{g}'$ is given by
\begin{equation*}
	\tensor{(\tilde{\Gamma}')}{^K_I_J}
	=\tensor{\tilde{\Gamma}}{^K_I_J}-\frac{1}{2}\tensor{(\tilde{g}'^{-1})}{^K^L}\tensor{(DA)}{_L_I_J}
	=\tensor{\tilde{\Gamma}}{^K_I_J}-\frac{1}{2}\tensor{(DA)}{^K_I_J}+O(r^{n/2}),
\end{equation*}
where
\begin{equation*}
	\tensor{(DA)}{_K_I_J}
	=\tensor{\tilde{\nabla}}{_K}\tensor{A}{_I_J}-\tensor{\tilde{\nabla}}{_I}\tensor{A}{_K_J}
	-\tensor{\tilde{\nabla}}{_J}\tensor{A}{_K_I}.
\end{equation*}
Hence
\begin{equation*}
	\tensor{(\delta_{\tilde{g}'}\tilde{\sigma})}{_I}
	=\tensor{(\delta_{\tilde{g}}\tilde{\sigma})}{_I}
	+\frac{1}{2}\tensor{(DA)}{^J^K_I}\tensor{\tilde{\sigma}}{_J_K}
	+\frac{1}{2}\tensor{(DA)}{^J^K_K}\tensor{\tilde{\sigma}}{_I_J}+O(r^{n/2}).
\end{equation*}
Let $A=r^{n/2}\overline{A}$. Then
\begin{equation*}
	\tensor{(DA)}{_K_I_J}
	=nr^{n/2-1}(\tensor{T}{_K}\tensor{\overline{A}}{_I_J}
	-\tensor{T}{_I}\tensor{\overline{A}}{_K_J}
	-\tensor{T}{_J}\tensor{\overline{A}}{_K_I})+O(r^{n/2})
\end{equation*}
and, because $T\contraction\overline{A}=O^-(r)$,
\begin{equation*}
	\tensor{(DA)}{_K_I^I}
	=nr^{n/2-1}\tensor{T}{_K}\tensor{\overline{A}}{_I^I}+O^-(r^{n/2}).
\end{equation*}
Therefore, if $\tilde{\sigma}\in\tilde{\caS}^X_\mathrm{aTT}(w)$,
$\divergence_{\tilde{g}'}\tilde{\sigma}=\divergence_{\tilde{g}}\tilde{\sigma}+O^-(r^{n/2})
=O^-(r^{\ceil{(n-2+w)/2}})$.

\begin{lem}
	\label{lem:AmbientLift}
	Let $\tilde{g}$ be an ambient metric and $\varphi\in\caS_0(-n/2+2+k)$, where $k\in\bbZ_+$.
	If $n$ is odd, then there exists $\tilde{\sigma}\in\tilde{\caS}_\mathrm{TT}^X(-n/2+2+k)$ such that
	$\tilde{\sigma}|_{TM}=\varphi$.
	If $n$ is even, there exists $\tilde{\sigma}\in\tilde{\caS}_\mathrm{aTT}^X(-n/2+2+k)$ such that
	$\tilde{\sigma}|_{TM}=\varphi$ as long as $k\le n/2$.
	In both cases, the restriction $\tilde{\varphi}=\tilde{\sigma}|_\caG$ is uniquely determined.
\end{lem}

\begin{proof}
	To prove the existence part,
	take any $\tilde{\sigma}_{(0)}\in\tilde{\caS}^X_0(-n/2+2+k)$ for which $\tilde{\sigma}_{(0)}|_{TM}=\varphi$.
	We shall inductively construct $\tilde{\sigma}_{(m)}\in\tilde{\caS}^X_0(-n/2+2+k)$
	for nonnegative integers $m$ such that
	\begin{equation*}
		\tilde{\sigma}_{(m)}=\tilde{\sigma}_{(m-1)}+O(r^{m-1}),\qquad
		\divergence_{\tilde{g}}\tilde{\sigma}_{(m)}=O(r^m).
	\end{equation*}
	Suppose we have $\tilde{\sigma}_{(m-1)}\in\tilde{\caS}^X_0(-n/2+2+k)$ with
	$\divergence_{\tilde{g}}\tilde{\sigma}_{(m-1)}=O(r^{m-1})$.
	If $\tilde{\sigma}_{(m)}\in\tilde{\caS}^X_0(-n/2+2+k)$,
	then $T\contraction\divergence_{\tilde{g}}\tilde{\sigma}_{(m)}=0$ is automatically guaranteed:
	\begin{equation*}
		\begin{split}
			\tensor{T}{^I}\tensor{\tilde{\nabla}}{^J}\tensor{(\tilde{\sigma}_{(m-1)})}{_I_J}
			&=\tensor{\tilde{\nabla}}{^J}(\tensor{T}{^I}\tensor{(\tilde{\sigma}_{(m-1)})}{_I_J})
			-(\tensor{\tilde{\nabla}}{^J}\tensor{T}{^I})\tensor{(\tilde{\sigma}_{(m-1)})}{_I_J}\\
			&=0+\tensor{\tilde{g}}{^I^J}\tensor{(\tilde{\sigma}_{(m-1)})}{_I_J}=0.
		\end{split}
	\end{equation*}
	We seek for $\tilde{\sigma}_{(m)}$ assuming that it is of the form
	\begin{equation}
		\label{eq:TTExtension_FormOfInduction}
		\tensor{(\tilde{\sigma}_{(m)})}{_I_J}
		=\tensor{(\tilde{\sigma}_{(m-1)})}{_I_J}+2r^{m-1}\tensor{T}{_(_I}\tensor{V}{_J_)}
		+r^{m-1}\tilde{f}\tensor{T}{_I}\tensor{T}{_J}-r^m\tensor{W}{_I_J},
	\end{equation}
	where $V\in\tilde{\caT}(-n/2+2+k-2m)$ satisfies $\tensor{T}{^I}\tensor{V}{_I}=0$,
	$\tilde{f}\in\tilde{\caE}(-n/2+k-2m)$, and $W\in\tilde{\caS}^X(-n/2+2+k-2m)$ is such that
	the whole expression \eqref{eq:TTExtension_FormOfInduction} is trace-free and vanishes if contracted with $T$
	(hence $\tr_{\tilde{g}}W=\tilde{f}$, $\tensor{T}{^J}\tensor{W}{_I_J}=\tensor{V}{_I}+\tilde{f}\tensor{T}{_I}$).
	Minus of the divergences of the additional three terms
	on the right-hand side of \eqref{eq:TTExtension_FormOfInduction} are
	\begin{align*}
		\tensor{\tilde{\nabla}}{^J}(2r^{m-1}\tensor{T}{_(_I}\tensor{V}{_J_)})
		&=r^{m-1}\cdot((n/2+2+k)\tensor{V}{_I}+\tensor{T}{_I}\tensor{\tilde{\nabla}}{^J}\tensor{V}{_J})+O(r^m),\\
		\tensor{\tilde{\nabla}}{^J}(r^{m-1}\tilde{f}\tensor{T}{_I}\tensor{T}{_J})
		&=r^{m-1}\cdot(n/2+1+k)\tilde{f}\tensor{T}{_I}+O(r^m),\\
		\tensor{\tilde{\nabla}}{^J}(-r^m\tensor{W}{_I_J})
		&=r^{m-1}\cdot(-2m)(\tensor{V}{_I}+\tilde{f}\tensor{T}{_I})+O(r^m).
	\end{align*}
	Therefore, we first put $V=(n/2+2+k-2m)^{-1}r^{-m+1}\divergence_{\tilde{g}}\tilde{\sigma}_{(m-1)}$,
	and set $\tilde{f}=-(n/2+1+k-2m)^{-1}\tensor{\tilde{\nabla}}{^J}\tensor{V}{_J}$
	so that the $O(r^{m-1})$-term of the divergence of \eqref{eq:TTExtension_FormOfInduction} vanishes.
	This is possible for all $m$ if $n$ is odd,
	and until $m=\lfloor n/2+k\rfloor$ if $n$ is even.
	Applying Borel's Lemma, the proof of the existence for $n$ odd is complete.
	When $n$ is even, we get $\tilde{\sigma}=\tilde{\sigma}_{(\lfloor(n/2+k)/2\rfloor)}$.
	Furthermore, if $n/2+1+k$ is an even number, then $\divergence_{\tilde{g}}\tilde{\sigma}$
	can be made $O^-(r^{(n/2+1+k)/2})$.
	Anyway, $\divergence_{\tilde{g}}\tilde{\sigma}$ finally becomes $O^-(r^{\lceil(n/2+k)/2\rceil})$,
	and the existence for $n$ even is proved.

	Let us once again take $\tilde{\sigma}_{(0)}$ as we did in the beginning of this proof.
	If $\tilde{\sigma}$ is as in the statement, then since $(T\contraction\tilde{\sigma})|_\caG=0$ and
	$\tilde{\sigma}|_{TM}=\varphi$, $\tilde{\sigma}$ must be written as
	\begin{equation*}
		\tilde{\sigma}=\tilde{\sigma}_{(0)}+2\tensor{T}{_(_I}\tensor{V}{_J_)}-r\tensor{W}{_I_J},
	\end{equation*}
	where $\tensor{T}{^I}\tensor{V}{_I}=O(r)$.
	Moreover, in order $T\contraction\tilde{\sigma}=O(r^2)$ to be satisfied,
	$\tensor{T}{^J}\tensor{W}{_I_J}$ should be
	$\tensor{V}{_I}+r^{-1}\tensor{T}{_I}\tensor{T}{^J}\tensor{V}{_J}+O(r)$.
	Then
	\begin{equation*}
		\tensor{\tilde{\nabla}}{^J}(2\tensor{T}{_(_I}\tensor{V}{_J_)}-r\tensor{W}{_I_J})
		=\left(\frac{n}{2}+k\right)\tensor{V}{_I}
		+\tensor{T}{_I}(\tensor{\tilde{\nabla}}{^J}\tensor{V}{_J}-2r^{-1}\tensor{T}{^J}\tensor{V}{_J})+O(r).
	\end{equation*}
	Therefore, $\tensor{V}{_I}$ mod $O^-(r)$ is determined by the condition
	$\divergence_{\tilde{g}}\tilde{\sigma}=O^-(r)$.
	If we put $\tilde{f}\tensor{T}{_I}$ into $\tensor{V}{_I}$, then the right-hand side will be
	$(n+2k-2)\tilde{f}\tensor{T}{_I}$.
	Thus $\tensor{V}{_I}$ is uniquely determined in order to satisfy $\divergence_{\tilde{g}}\tilde{\sigma}=O(r)$.
\end{proof}

We call $\tilde{\varphi}$ in Lemma \ref{lem:AmbientLift} the \emph{ambient lift} of $\varphi\in\caS(-n/2+2+k)$.

\section{A GJMS construction for trace-free symmetric 2-tensors}
\label{sec:GJMS}

Let $(M,[g])$ be a conformal manifold of dimension $n\ge 3$ and $\tilde{g}$ an ambient metric.
We shall play with the following three operators:
\begin{align*}
	x &\colon \tilde{\caS}(w)\to\tilde{\caS}(w+2), & \tilde{\sigma}&\mapsto\frac{1}{4}r\tilde{\sigma},\\
	y &\colon \tilde{\caS}(w)\to\tilde{\caS}(w-2), &
	\tilde{\sigma}&\mapsto\ambientLichnerowicz\tilde{\sigma},\\
	h &\colon \tilde{\caS}(w)\to\tilde{\caS}(w), &
	\tilde{\sigma}&\mapsto\left(\tilde{\nabla}_T+\frac{n+2}{2}\right)\tilde{\sigma}
	=\left(w+\frac{n}{2}-1\right)\tilde{\sigma}.
\end{align*}
Just as in the case of the classical GJMS construction, one can verify the following.

\begin{prop}
	The operators $x$, $y$, $h$ enjoy the $\mathfrak{sl}(2)$ commutation relations:
	\begin{equation*}
		[h,x]=2x,\qquad
		[h,y]=-2y,\qquad
		[x,y]=h.
	\end{equation*}
\end{prop}

The proof is left to the reader.
Consequently we have the following identities:
\begin{align}
	\label{eq:SL2Commutations1}
	[y^m,x]&=-my^{m-1}(h-m+1),\\
	\label{eq:SL2Commutations2}
	[x^m,y]&=mx^{m-1}(h+m-1),\\
	\label{eq:SL2Commutations3}
	y^{m-1}x^{m-1}&=(-1)^{m-1}(m-1)!h(h+1)\cdots(h+m-2)+xZ_m,
\end{align}
where $Z_m$ is some polynomial in $x$, $y$, $h$.

We are going to verify that $x$, $y$, and $h$ preserve the subspaces $\tilde{\caS}^X_\mathrm{TT}(w)$ when
$n$ is odd and $\tilde{\caS}^X_\mathrm{aTT}(w)$ when $n$ even.
For this we need two lemmas.

\begin{lem}
	For $\tilde{f}\in\tilde{\caE}(w)$, $\tilde{\tau}\in\tilde{\caT}(w)$,
	\begin{gather}
		\label{eq:LaplacianOnFunctions}
		\tilde{f}=O(r^m)\Longrightarrow\tilde{\Delta}\tilde{f}=O(r^{m-1}),\\
		\label{eq:LaplacianOnForms}
		\tilde{\tau}=O^-(r^m)\Longrightarrow\tilde{\Delta}\tilde{\tau}=O^-(r^{m-1}).
	\end{gather}
	In \eqref{eq:LaplacianOnForms}, we may also replace $\tilde{\Delta}$ with the Hodge Laplacian $\ambientHodge$.
\end{lem}

\begin{proof}
	First we compute $\tilde{\Delta}(r^m)$:
	\begin{equation*}
		\tilde{\Delta}(r^m)=-\tensor{\tilde{\nabla}}{^I}\tensor{\tilde{\nabla}}{_I}(r^m)
		=-\tensor{\tilde{\nabla}}{^I}(2mr^{m-1}\tensor{T}{_I})
		=-2m(2m+n)r^{m-1}.
	\end{equation*}
	Hence it is clear that $\tilde{f}=O(r^m)$ implies $\tilde{\Delta}\tilde{f}=O(r^{m-1})$ and that
	$\tilde{\tau}=O(r^m)$ implies $\tilde{\Delta}\tilde{\tau}=O(r^{m-1})$.
	So, to prove \eqref{eq:LaplacianOnForms},
	it remains to show that $\tilde{\Delta}(r^{m-1}\tilde{f}\tensor{T}{_I})$ is $O^-(r^{m-1})$.
	This is checked directly:
	\begin{equation*}
		\tensor{\tilde{\nabla}}{_J}(r^{m-1}\tilde{f}\tensor{T}{_I})
		=2(m-1)r^{m-2}\tilde{f}\tensor{T}{_I}\tensor{T}{_J}+r^{m-1}\tilde{f}\tensor{\tilde{g}}{_I_J}
		+r^{m-1}\tensor{T}{_I}\tensor{\tilde{\nabla}}{_J}\tilde{f}
	\end{equation*}
	and therefore
	\begin{equation*}
		\tilde{\Delta}(r^{m-1}\tilde{f}\tensor{T}{_I})
		=-2(m-1)(2m+n+2w)r^{m-2}\tilde{f}\tensor{T}{_I}+O(r^{m-1}).
	\end{equation*}
	By Bochner's Formula
	$\ambientHodge\tensor{\tilde{\tau}}{_I}=\tilde{\Delta}\tensor{\tilde{\tau}}{_I}
	+\tensor{\widetilde{\Ric}}{_I^J}\tensor{\tilde{\tau}}{_J}$,
	$\ambientHodge\tilde{\tau}=O^-(r^{m-1})$ is clear.
\end{proof}

Let $(D\widetilde{\Ric})^\circ\colon\tilde{\caS}(w)\to\tilde{\caT}(w-4)$ be defined by
\begin{equation*}
	\tensor{((D\widetilde{\Ric})^\circ\tilde{\sigma})}{_I}
	=(\tensor{\tilde{\nabla}}{_I}\tensor{\widetilde{\Ric}}{_J_K}
	-\tensor{\tilde{\nabla}}{_J}\tensor{\widetilde{\Ric}}{_I_K}
	-\tensor{\tilde{\nabla}}{_K}\tensor{\widetilde{\Ric}}{_I_J})\tensor{\tilde{\sigma}}{^J^K}.
\end{equation*}
Then it is known that, on any symmetric 2-tensor,
\begin{equation}
	\label{eq:CommutatorDivLichnerowicz}
	\divergence_{\tilde{g}}\circ\ambientLichnerowicz=\ambientHodge\circ\divergence_{\tilde{g}}
	+(D\widetilde{\Ric})^\circ.
\end{equation}

\begin{lem}
	\label{lem:DRic}
	When $n$ is even and $2-n\le w\le 2$,
	\begin{equation*}
		\tilde{\sigma}\in\tilde{\caS}^X_\mathrm{aTT}(w)\Longrightarrow
		(D\widetilde{\Ric})^\circ\tilde{\sigma}=O^-(r^{n/2-1}).
	\end{equation*}
\end{lem}

\begin{proof}
	Let $\widetilde{\Ric}=r^{n/2-1}\tilde{S}$. Then
	\begin{equation}
		\label{eq:DerivativeOfAmbientRic}
		\tensor{\tilde{\nabla}}{_I}\tensor{\widetilde{\Ric}}{_J_K}
		=(n-2)r^{n/2-2}\tensor{T}{_I}\tensor{\tilde{S}}{_J_K}+O(r^{n/2-1}).
	\end{equation}
	Therefore
	\begin{equation*}
		(\tensor{\tilde{\nabla}}{_I}\tensor{\widetilde{\Ric}}{_J_K})\tensor{\tilde{\sigma}}{^J^K}
		=(n-2)r^{n/2-2}\braket{\tilde{S},\tilde{\sigma}}_{\tilde{g}}\tensor{T}{_I}+O(r^{n/2-1}).
	\end{equation*}
	On the other hand, since $T\contraction\tilde{\sigma}$ is at least $O^-(r)$,
	we can write $\tensor{T}{^I}\tensor{\tilde{\sigma}}{_I_J}=\tilde{f}\tensor{T}{_J}+O(r)$.
	Hence, by \eqref{eq:DerivativeOfAmbientRic} and \eqref{eq:TAmbientCurvature},
	\begin{equation*}
		(\tensor{\tilde{\nabla}}{_J}\tensor{\widetilde{\Ric}}{_I_K})\tensor{\tilde{\sigma}}{^J^K}
		=(n-2)r^{n/2-2}\tilde{f}\tensor{T}{^K}\tensor{\tilde{S}}{_I_K}+O(r^{n/2-1})
		=O(r^{n/2-1}).
	\end{equation*}
	Consequently, $(D\widetilde{\Ric})^\circ\tilde{\sigma}=O^-(r^{n/2-1})$.
\end{proof}

\begin{prop}
	\label{prop:MappingPropertyOfSL2}
	If $n$ is odd, then for any $w$,
	\begin{gather*}
		x(\tilde{\caS}_\mathrm{TT}^X(w))\subset\tilde{\caS}_\mathrm{TT}^X(w+2),\quad
		y(\tilde{\caS}_\mathrm{TT}^X(w))\subset\tilde{\caS}_\mathrm{TT}^X(w-2),\quad
		h(\tilde{\caS}_\mathrm{TT}^X(w))\subset\tilde{\caS}_\mathrm{TT}^X(w).
	\end{gather*}
	If $n$ is even,
	\begin{alignat*}{2}
		x(\tilde{\caS}_\mathrm{aTT}^X(w))&\subset\tilde{\caS}_\mathrm{aTT}^X(w+2),&\qquad &2-n\le w\le 0,\\
		y(\tilde{\caS}_\mathrm{aTT}^X(w))&\subset\tilde{\caS}_\mathrm{aTT}^X(w-2),& &{-n}\le w\le 2,\\
		h(\tilde{\caS}_\mathrm{aTT}^X(w))&\subset\tilde{\caS}_\mathrm{aTT}^X(w),& &2-n\le w\le 2.
	\end{alignat*}
\end{prop}

\begin{proof}
	Since the case $n$ odd is easier to prove, we discuss the case $n$ even.
	It is clear that $h(\tilde{\caS}_\mathrm{aTT}^X(w))\subset\tilde{\caS}_\mathrm{aTT}^X(w)$.
	For $\tilde{\sigma}\in\tilde{\caS}_\mathrm{aTT}^X(w)$,
	$T\contraction(r\tilde{\sigma})=rT\contraction\tilde{\sigma}=O^-(r^{\ceil{(n-2+w)/2}+2})$,
	$\tr_{\tilde{g}}(r\tilde{\sigma})=r\tr_{\tilde{g}}\tilde{\sigma}=O(r^{\ceil{(n-2+w)/2}+1})$,
	and
	\begin{equation*}
		\divergence_{\tilde{g}}(r\tilde{\sigma})
		=-2T\contraction\tilde{\sigma}+r\divergence_{\tilde{g}}\tilde{\sigma}
		=O^-(r^{\ceil{\frac{n-2+w}{2}}+1}).
	\end{equation*}
	Hence $x\tilde{\sigma}\in\tilde{\caS}_\mathrm{aTT}^X(w+2)$.
	It remains to show that $y\tilde{\sigma}\in\tilde{\caS}_\mathrm{aTT}^X(w-2)$.
	The trace of $\ambientLichnerowicz\tilde{\sigma}$ is
	$\tr_{\tilde{g}}\ambientLichnerowicz\tilde{\sigma}=\tilde{\Delta}(\tr_{\tilde{g}}\tilde{\sigma})
	=O(r^{\ceil{(n-2+w)/2}-1})$ by \eqref{eq:LaplacianOnFunctions}.
	Furthermore,
	\begin{equation*}
		\tensor{\tilde{\nabla}}{_K}(\tensor{T}{^J}\tensor{\tilde{\sigma}}{_I_J})
		=\tensor{\delta}{_K^J}\tensor{\tilde{\sigma}}{_I_J}
		+\tensor{T}{^J}\tensor{\tilde{\nabla}}{_K}\tensor{\tilde{\sigma}}{_I_J}
		=\tensor{\tilde{\sigma}}{_I_K}+\tensor{T}{^J}\tensor{\tilde{\nabla}}{_K}\tensor{\tilde{\sigma}}{_I_J}
	\end{equation*}
	and hence
	\begin{equation*}
		\begin{split}
			\tilde{\Delta}(\tensor{T}{^J}\tensor{\tilde{\sigma}}{_I_J})
			&=-\tensor{\tilde{\nabla}}{^K}\tensor{\tilde{\sigma}}{_I_K}
			-\tensor{\tilde{\nabla}}{^K}(\tensor{T}{^J}\tensor{\tilde{\nabla}}{_K}\tensor{\tilde{\sigma}}{_I_J})\\
			&=-2\tensor{\tilde{\nabla}}{^K}\tensor{\tilde{\sigma}}{_I_K}
			-\tensor{T}{^J}\tensor{\tilde{\nabla}}{^K}\tensor{\tilde{\nabla}}{_K}\tensor{\tilde{\sigma}}{_I_J}
			=-2\tensor{\tilde{\nabla}}{^K}\tensor{\tilde{\sigma}}{_I_K}
			+\tensor{T}{^J}\ambientLichnerowicz\tensor{\tilde{\sigma}}{_I_J};
		\end{split}
	\end{equation*}
	the last equality is because of \eqref{eq:TAmbientCurvature}.
	This implies $T\contraction\ambientLichnerowicz\tilde{\sigma}=O^-(r^{\ceil{(n-2+w)/2}})$.
	Finally, \eqref{eq:CommutatorDivLichnerowicz} and Lemma \ref{lem:DRic} show
	$\divergence_{\tilde{g}}\ambientLichnerowicz\tilde{\sigma}=O^-(r^{\ceil{(n-2+w)/2}-1})$.
\end{proof}

\begin{thm}
	\label{thm:GJMSAsLaplacianPower}
	Let $k\in\bbZ_+$ if $n$ odd, and $k\in\set{1,2,\dots,n/2}$ if $n$ even.
	For any $\varphi\in\caS_0(-n/2+2+k)$, let $\tilde{\sigma}\in\tilde{\caS}(-n/2+2+k)$ be
	any extension of the ambient lift $\tilde{\varphi}$.
	Then $\ambientLichnerowicz^k\tilde{\sigma}|_\caG$ depends only on $\varphi$ and not on the extension.
	Furthermore, $\ambientLichnerowicz^k\tilde{\sigma}|_{TM}$ makes sense as a section in $\caS(-n/2+2-k)$.
\end{thm}

\begin{proof}
	We work on the case $n$ even only.
	Any two extensions of $\tilde{\varphi}$ differs by a tensor of the form $r\tilde{\tau}$,
	where $\tilde{\tau}\in\tilde{\caS}_0(-n/2+k)$.
	Equation \eqref{eq:SL2Commutations1} shows that the commutator $[\ambientLichnerowicz^k,r]$ vanishes on
	$\tilde{\caS}_0(-n/2+k)$ and hence $\ambientLichnerowicz^k(r\tilde{\tau})|_\caG=0$.
	In particular, using Lemma \ref{lem:AmbientLift} one can take
	$\tilde{\sigma}\in\tilde{\caS}^X_\mathrm{aTT}(-n/2+2+k)$ as an extension of $\tilde{\varphi}$.
	Then by Proposition \ref{prop:MappingPropertyOfSL2},
	$\ambientLichnerowicz^k\tilde{\sigma}\in\tilde{\caS}^X_\mathrm{aTT}(-n/2+2-k)$
	and $\ambientLichnerowicz^k\tilde{\sigma}|_{TM}$ is defined.
\end{proof}

\begin{thm}
	\label{thm:GJMSAsObstruction}
	Let $k\in\bbZ_+$ if $n$ odd, and $k\in\set{1,2,\dots,n/2}$ if $n$ even.
	Let $\varphi\in\caS_0(-n/2+2+k)$ and $\tilde{\varphi}$ its ambient lift.
	Then there exists a solution
	$\tilde{\sigma}\in\tilde{\caS}^X_\mathrm{TT}(-n/2+2+k)$ if $n$ odd,
	and $\tilde{\sigma}\in\tilde{\caS}^X_\mathrm{aTT}(-n/2+2+k)$ if $n$ even, to the problem
	\begin{equation}
		\label{eq:HarmonicExtensionProblem}
		\ambientLichnerowicz\tilde{\sigma}=O(r^{k-1}),\qquad
		\tilde{\sigma}|_\caG=\tilde{\varphi},
	\end{equation}
	which is unique modulo $O(r^k)$.
	For any such $\tilde{\sigma}$, $(r^{1-k}\ambientLichnerowicz\tilde{\sigma})|_{\caG}$
	is independent of the ambiguity that lives in $\tilde{\sigma}$,
	and agrees with $\ambientLichnerowicz^k\tilde{\sigma}|_{\caG}$ up to a constant factor:
	\begin{equation}
		\label{eq:EqualityOfTwoOperators}
		(r^{1-k}\ambientLichnerowicz\tilde{\sigma})|_{\caG}
		=\frac{1}{4^{k-1}(k-1)!^2}\ambientLichnerowicz^k\tilde{\sigma}|_{\caG}.
	\end{equation}
\end{thm}

\begin{proof}
	We work on the case $n$ even only.
	Let us begin with an arbitrary extension $\tilde{\sigma}_{(0)}\in\tilde{\caS}^X_\mathrm{aTT}(-n/2+2+k)$
	of $\tilde{\varphi}$.
	If an extension $\tilde{\sigma}_{(m-1)}$ satisfies
	$\ambientLichnerowicz\tilde{\sigma}_{(m-1)}=O(r^{m-1})$,
	then it has a modification $\tilde{\sigma}_{(m)}=\tilde{\sigma}_{(m-1)}+r^m\tilde{\sigma}_1$,
	$\tilde{\sigma}_1\in\tilde{\caS}^X_\mathrm{TT}(-n/2+2+k-2m)$, which is unique modulo $O(r^{m+1})$,
	satisfying $\ambientLichnerowicz\tilde{\sigma}_{(m)}=O(r^m)$.
	In fact, by \eqref{eq:SL2Commutations2}, we have
	\begin{equation}
		\label{eq:AmbLichnerowiczAndrm}
		\ambientLichnerowicz(r^m\tilde{\sigma}_1)
		=4mr^{m-1}(m-k)\tilde{\tau}+r^m\ambientLichnerowicz\tilde{\sigma}_1.
	\end{equation}
	Therefore $\tilde{\sigma}_1$ can be taken so that
	$\ambientLichnerowicz\tilde{\sigma}_{(m)}=O(r^m)$ unless $m=k$.
	Hence there is $\tilde{\sigma}$ with the property stated in the theorem.
	Let $\ambientLichnerowicz\tilde{\sigma}=r^{k-1}\tilde{F}$, $\tilde{F}\in\tilde{\caS}_\mathrm{TT}^X(-n/2+2-k)$.
	Then, by \eqref{eq:SL2Commutations3},
	$\ambientLichnerowicz^k\tilde{\sigma}=4^{k-1}y^{k-1}x^{k-1}\tilde{F}=4^{k-1}(k-1)!^2\tilde{F}+O(r)$.
	Hence \eqref{eq:EqualityOfTwoOperators}.
\end{proof}

Except in the case where $n$ is even and $k=n/2$,
$(\ambientLichnerowicz^k\tilde{\sigma})|_{TM}$ is trace-free since
$\tr_{\tilde{g}}\ambientLichnerowicz^k\tilde{\sigma}$ and
$T\contraction\ambientLichnerowicz^k\tilde{\sigma}$ are both $O(r)$.

\begin{dfn}
	Let $(M,[g])$ be a conformal manifold of dimension $n\ge 3$ and $\tilde{g}$ an ambient metric.
	We call
	\begin{equation*}
		P_k\colon\caS_0(-n/2+2+k)\to\caS_0(-n/2+2-k),\qquad
		P_k\varphi=\tf_{\bm{g}}(\ambientLichnerowicz^k\tilde{\sigma}|_{TM})
	\end{equation*}
	the \emph{GJMS operator on trace-free symmetric 2-tensors},
	where $\tilde{\sigma}\in\tilde{\caS}(-n/2+2+k)$ is any extension of the ambient lift of $\varphi$.
	(One can remove $\tf_{\bm{g}}$ unless $n$ is even and $k=n/2$.)
	In particular, when $n=\dim M\ge 4$ even,
	\begin{equation*}
		P=P_{n/2}\colon\caS_0(2)\to\caS_0(2-n)
	\end{equation*}
	is called the \emph{critical GJMS operator on trace-free symmetric 2-tensors}.
\end{dfn}

\begin{thm}
	\label{thm:InvarianceOfGJMS}
	The GJMS operators on trace-free symmetric 2-tensors do not depend on the choice of $\tilde{g}$,
	and hence are conformally invariant differential operators.
\end{thm}

For the case where $n$ is even and $k=n/2$, the direct verification of the conformal invariance is not easy.
We will see in Theorem \ref{thm:DifferentialOfObstruction} that, up to a constant factor,
$P\varphi$ is equal to $\tf_{\bm{g}}\caO'_{\bm{g}}\varphi$,
which is clearly conformally invariant.
Here, we prove the theorem in the case $n$ odd and the case $n$ even, $k\le n/2-1$.

\begin{proof}[Proof of Theorem \ref{thm:InvarianceOfGJMS} except the case where $(n,k)=(\text{even},n/2)$]
	By Theorem \ref{thm:GJMSAsObstruction}, we may work with $r^{1-k}\ambientLichnerowicz\tilde{\sigma}$
	instead of $\ambientLichnerowicz^k\tilde{\sigma}$.
	Let $\tilde{g}$ be an ambient metric, $\varphi\in\caS_0(-n/2+2+k)$ and $\tilde{\sigma}$
	a solution to the problem stated in Theorem \ref{thm:GJMSAsObstruction}.
	Then, if $\Phi$ is an ambient-equivalence map, $\Phi^*\tilde{\sigma}$ solves the same problem
	with respect to $\Phi^*\tilde{g}$.
	Since
	$(\Phi^*r)^{1-k}\tilde{\Delta}_{\mathrm{L},\Phi^*\tilde{g}}(\Phi^*\tilde{\sigma})
	=\Phi^*(r^{1-k}\ambientLichnerowicz\sigma)$,
	the restrictions of $(\Phi^*r)^{1-k}\tilde{\Delta}_{\mathrm{L},\Phi^*\tilde{g}}(\Phi^*\tilde{\sigma})$
	and $r^{1-k}\ambientLichnerowicz\sigma$ to $TM$ coincide.
	Therefore we may assume that $\tilde{g}$ is in normal form.

	When $n$ is odd, the assertion is now clear because $\tilde{g}$ is formally unique if it is in normal form.
	So we assume that $n$ is even in what follows.
	It suffices to show that, if $\tilde{g}$, $\Hat{\tilde{g}}$ are ambient metrics in normal form
	and $\tilde{\sigma}\in\tilde{\caS}^X_\mathrm{aTT}(-n/2+2+k)$,
	\begin{equation*}
		\Hat{\tilde{\Delta}}_\mathrm{L}\tilde{\sigma}-\ambientLichnerowicz\tilde{\sigma}=O(\rho^{n/2-2})\qquad
		\text{and}\qquad
		\Hat{\tilde{\Delta}}_\mathrm{L}\tensor{\tilde{\sigma}}{_i_j}
		-\ambientLichnerowicz\tensor{\tilde{\sigma}}{_i_j}=O(\rho^{n/2-1}).
	\end{equation*}
	Let $\tensor{D}{^K_I_J}=\tensor{\Hat{\tilde{\Gamma}}}{^K_I_J}-\tensor{\tilde{\Gamma}}{^K_I_J}$.
	From~\cite[Equation (3.16)]{Fefferman:2012vr}, one concludes that
	$\tensor{D}{^K_I_J}=O(\rho^{n/2-1})$ and $\tensor{\tilde{\nabla}}{^I}\tensor{D}{^K_I_J}=O(\rho^{n/2-1})$.
	Therefore
	\begin{equation*}
		\Hat{\tilde{\Delta}}\tensor{\tilde{\sigma}}{_I_J}-\tilde{\Delta}\tensor{\tilde{\sigma}}{_I_J}
		=\tensor{\tilde{\nabla}}{^K}(2\tensor{D}{^L_K_(_I}\tensor{\tilde{\sigma}}{_J_)_L})+O(\rho^{n/2-1})
		=O(\rho^{n/2-1}).
	\end{equation*}
	In addition, $\widehat{\widetilde{\Ric}}=\widetilde{\Ric}+O(\rho^{n/2-1})$ and
	$\Hat{\tilde{R}}=\tilde{R}+O(\rho^{n/2-2})$ by~\cite[Equation (6.1)]{Fefferman:2012vr}.
	Hence
	$\Hat{\tilde{\Delta}}_\mathrm{L}\tilde{\sigma}-\ambientLichnerowicz\tilde{\sigma}=O(\rho^{n/2-2})$.
	Moreover, if
	$\tensor{\tilde{S}}{_I_J_K_L}=\tensor{\Hat{\tilde{R}}}{_I_J_K_L}-\tensor{\tilde{R}}{_I_J_K_L}$, then
	\begin{equation*}
		\begin{split}
			\Hat{\tilde{\Delta}}_\mathrm{L}\tensor{\tilde{\sigma}}{_i_j}
			-\ambientLichnerowicz\tensor{\tilde{\sigma}}{_i_j}
			&=-2t^{-4}\tensor{(g^{-1}_\rho)}{^k^m}\tensor{(g^{-1}_\rho)}{^l^n}
			\tensor{\tilde{S}}{_i_k_j_l}\tensor{\tilde{\sigma}}{_m_n}\\
			&\phantom{=\;}-4t^{-3}\tensor{(g^{-1}_\rho)}{^k^m}
			\tensor{\tilde{S}}{_i_k_j_\infty}\tensor{\tilde{\sigma}}{_m_0}
			+4t^{-4}\rho\tensor{(g^{-1}_\rho)}{^k^m}
			\tensor{\tilde{S}}{_i_k_j_\infty}\tensor{\tilde{\sigma}}{_m_\infty}\\
			&\phantom{=\;}-2t^{-2}
			\tensor{\tilde{S}}{_i_\infty_j_\infty}\tensor{\tilde{\sigma}}{_0_0}
			+4t^{-3}\rho
			\tensor{\tilde{S}}{_i_\infty_j_\infty}\tensor{\tilde{\sigma}}{_0_\infty}
			-8t^{-4}\rho^2
			\tensor{\tilde{S}}{_i_\infty_j_\infty}\tensor{\tilde{\sigma}}{_\infty_\infty}\\
			&\phantom{=\;}+O(\rho^{n/2-1}).
		\end{split}
	\end{equation*}
	Again by~\cite[Equation (6.1)]{Fefferman:2012vr}, we have
	$\tensor{\tilde{S}}{_i_j_k_l}=O(\rho^{n/2-1})$, $\tensor{\tilde{S}}{_i_j_k_\infty}=O(\rho^{n/2-1})$,
	$\tensor{\tilde{S}}{_i_\infty_k_\infty}=O(\rho^{n/2-2})$ and hence
	$\Hat{\tilde{\Delta}}_\mathrm{L}\tensor{\tilde{\sigma}}{_i_j}
	-\ambientLichnerowicz\tensor{\tilde{\sigma}}{_i_j}=O(\rho^{n/2-1})$.
\end{proof}

We close this section with a lemma that is proved just like the construction of $\tilde{\sigma}$ in
Theorem \ref{thm:GJMSAsObstruction}.

\begin{lem}
	\label{lem:SolvingLaplacian}
	Let $k\in\bbZ_+$.
	For any $\tilde{f}_1\in\tilde{\caE}(-n/2-2+k)$, there exists $\tilde{f}\in\tilde{\caE}(-n/2+k)$ such that
	\begin{equation*}
		\tilde{\Delta}\tilde{f}=\tilde{f}_1+O(r^{k-1}).
	\end{equation*}
	Likewise, for any $\tilde{\tau}_1\in\tilde{\caT}(-n/2-1+k)$, there exists
	$\tilde{\tau}\in\tilde{\caT}(-n/2+1+k)$ such that
	\begin{equation*}
		\tilde{\Delta}\tilde{\tau}=\tilde{\tau}_1+O(r^{k-1}).
	\end{equation*}
	In both problems, we may arbitrarily prescribe the values along $\caG$;
	if we prescribe $\tilde{f}|_\caG$, $\tilde{\tau}|_\caG$,
	then $\tilde{f}$, $\tilde{\tau}$ are unique modulo $O(r^k)$.
\end{lem}

\section{The variations of obstruction tensor and $Q$-curvature}
\label{sec:Variations}

Let $\tilde{g}$ be an ambient metric for a conformal manifold $(M,[g])$ of dimension $n\ge 3$.
Recall that, from general calculations on (pseudo-)Riemannian curvature tensors,
the differential of the Ricci tensor operator (which we write $\Ric$ here) is
\begin{equation}
	\label{eq:DifferentialRicci}
	\Ric'_{\tilde{g}}\tilde{\sigma}
	=\frac{1}{2}\ambientLichnerowicz\tilde{\sigma}-\divergence_{\tilde{g}}^*\caB_{\tilde{g}}\tilde{\sigma},
\end{equation}
where $\divergence_{\tilde{g}}^*$ is the dual of the divergence
$\tensor{(\divergence_{\tilde{g}}^*\tilde{\tau})}{_I_J}=\tensor{\tilde{\nabla}}{_(_I}\tensor{\tilde{\tau}}{_J_)}$
and $\caB_{\tilde{g}}$ is defined by
$\caB_{\tilde{g}}\tilde{\sigma}=\divergence_{\tilde{g}}\tilde{\sigma}+\frac{1}{2}d(\tr_{\tilde{g}}\tilde{\sigma})$.
Therefore, for $n$ even, a solution $\tilde{\sigma}\in\tilde{\caS}_\mathrm{aTT}^X(2)$
to the problem in Theorem \ref{thm:GJMSAsObstruction} approximately solves
$\Ric'_{\tilde{g}}\tilde{\sigma}=0$, and hence it is expected that we can read off $\caO'_{\bm{g}}\varphi$
from the asymptotics of $\tilde{\sigma}$.
This will finally turn out to be true,
but since the definition of $\caO$ depends on the existence theorem of normal-form ambient metrics,
in order to capture $\caO'_{\bm{g}}\varphi$ our starting point has to be infinitesimal modifications of
ambient metrics in normal form.
The differential equation that they (approximately) satisfy is different from
$\ambientLichnerowicz\tilde{\sigma}=0$.
So we shall establish a method for translating solutions of the two equations.

Let $(M,[g])$ be an $n$-dimensional conformal manifold with $n\ge 4$ even and $\varphi\in\caS_0(2)$.
Suppose that $\bm{g}_s$ is a family of conformal metrics
(here we use $s$ for the parameter, because $t$ will denote a coordinate on $\tilde{\caG}$)
with $\bm{g}_0=\bm{g}$ such that $\dot{\bm{g}}_s|_{s=0}=\varphi$.
Let $g\in[g]$ be any representative metric, and $g_s$ the corresponding representatives of $\bm{g}_s$.
By the method of~\cite{Fefferman:2012vr}, we can construct a family of ambient metrics
\begin{equation*}
	\tilde{g}_s=2\rho\,dt^2+2t\,dt\,d\rho+t^2g^s_\rho
\end{equation*}
such that $g_0^s=g_s$ and $g_\rho^s$ smoothly depends on the two variables $\rho$, $s$.
All these metrics satisfy $\widetilde{\Ric}_s=O(r^{n/2-1})$ and
$T\contraction\widetilde{\Ric}_s=O^-(r^{n/2})$.
Differentiating these equations, we conclude that
$\tilde{\sigma}=\tilde{\sigma}_\mathrm{norm}=(d/ds)\tilde{g}_s|_{s=0}$
solves
\begin{equation*}
	\Ric'_{\tilde{g}}\tilde{\sigma}=O(r^{n/2-1}),\qquad
	T\contraction\Ric'_{\tilde{g}}\tilde{\sigma}=O^-(r^{n/2}).
\end{equation*}
Note that it satisfies $T\contraction\tilde{\sigma}_\mathrm{norm}=0$,
$\tr_{\tilde{g}}\tilde{\sigma}_\mathrm{norm}=O(r)$, and hence
\begin{equation*}
	\tensor{T}{^I}\tensor{\tilde{\nabla}}{^J}\tensor{(\tilde{\sigma}_\mathrm{norm})}{_I_J}
	=\tensor{\tilde{\nabla}}{^J}(\tensor{T}{^I}\tensor{(\tilde{\sigma}_\mathrm{norm})}{_I_J})
	-\tensor{\tilde{g}}{^I^J}\tensor{(\tilde{\sigma}_\mathrm{norm})}{_I_J}
	=O(r);
\end{equation*}
therefore it holds that
\begin{equation}
	\label{eq:TContractionOfBianchiOfNormalSolution}
	T\contraction\caB_{\tilde{g}}\tilde{\sigma}_\mathrm{norm}
	=T\contraction\divergence_{\tilde{g}}\tilde{\sigma}_\mathrm{norm}
	+\frac{1}{2}T(\tr_{\tilde{g}}\tilde{\sigma}_\mathrm{norm})
	=O(r).
\end{equation}
Since the obstruction tensor $\caO=\caO_s$ is defined by
\begin{equation*}
	\caO_s=c_n(r^{1-n/2}\widetilde{\Ric}_s)|_{TM},\qquad
	c_n=(-1)^{n/2-1}\frac{2^{n-2}(n/2-1)!^2}{n-2},
\end{equation*}
we have
\begin{equation*}
	\caO'_{\bm{g}}\varphi=c_n(r^{1-n/2}\widetilde{\Ric}'_{\tilde{g}}\tilde{\sigma}_\mathrm{norm})|_{TM}.
\end{equation*}

\begin{lem}
	\label{lem:HarminizationOfNormalSolution}
	Let $\tilde{\sigma}_\mathrm{norm}$ be as above.
	Then, there exists a dilation-invariant vector field $\tilde{\xi}$ on $\tilde{\caG}$
	such that $\tilde{\xi}|_\caG=0$ and
	\begin{equation*}
		\caB_{\tilde{g}}(\tilde{\sigma}_\mathrm{norm}+\caK_{\tilde{g}}\tilde{\xi})=O(r^{n/2}),
	\end{equation*}
	where $\caK_{\tilde{g}}$ is the Killing operator:
	$\tensor{(\caK_{\tilde{g}}\tilde{\xi})}{_I_J}=2\tensor{\tilde{\nabla}}{_(_I}\tensor{\tilde{\xi}}{_J_)}$.
	Such $\tilde{\xi}$ is unique modulo $O(r^{n/2+1})$ and satisfies
	$\tilde{g}(T,\tilde{\xi})=O(r^2)$, $\tr_{\tilde{g}}\caK_{\tilde{g}}\tilde{\xi}=O(r)$.
\end{lem}

\begin{proof}
	The equation to be solved is
	$\caB_{\tilde{g}}\caK_{\tilde{g}}\tilde{\xi}=-\caB_{\tilde{g}}\tilde{\sigma}_\mathrm{norm}+O(r^{n/2})$.
	By a straightforward calculation,
	\begin{equation*}
		\tensor{(\caB_{\tilde{g}}\caK_{\tilde{g}}\tilde{\xi})}{_I}
		=\tilde{\Delta}\tensor{\tilde{\xi}}{_I}-\tensor{\widetilde{\Ric}}{_I_J}\tensor{\tilde{\xi}}{^J}.
	\end{equation*}
	Since $\tensor{\widetilde{\Ric}}{_I_J}\tensor{\tilde{\xi}}{^J}=O(r^{n/2})$ for any
	$\tilde{\xi}$ satisfying $\tilde{\xi}|_\caG=0$,
	the equation simplifies to
	$\tilde{\Delta}\tilde{\xi}=-\caB_{\tilde{g}}\tilde{\sigma}_\mathrm{norm}+O(r^{n/2})$.
	By Lemma \ref{lem:SolvingLaplacian}, $\tilde{\xi}$ is uniquely determined up to an $O(r^{n/2+1})$ ambiguity.

	If we write $\tilde{\xi}=rV$, then $\tilde{\Delta}\tilde{\xi}=-2nV+O(r)$. On the other hand,
	$T\contraction\tilde{\Delta}\tilde{\xi}=-2T\contraction\caB_{\tilde{g}}\tilde{\sigma}_\mathrm{norm}+O(r^{n/2})$
	should be $O(r)$ by \eqref{eq:TContractionOfBianchiOfNormalSolution}.
	Consequently $T\contraction V=O(r)$, i.e., $T\contraction\tilde{\xi}=O(r^2)$.
	Moreover,
	$\tr_{\tilde{g}}\caK_{\tilde{g}}\tilde{\xi}=2\tensor{\tilde{\nabla}}{^I}\tensor{\tilde{\xi}}{_I}
	=4\tensor{T}{^I}\tensor{V}{_I}+O(r)=O(r)$.
\end{proof}

Let $\tilde{\sigma}=\tilde{\sigma}_\mathrm{norm}+\caK_{\tilde{g}}\tilde{\xi}\in\tilde{\caS}(2)$.
It is a consequence of the fact that the Ricci operator commutes with diffeomorphisms that
$\Ric'_{\tilde{g}}\caK_{\tilde{g}}\tilde{\xi}=\Ric'_{\tilde{g}}\caL_{\tilde{\xi}}\tilde{g}
=\caL_{\tilde{\xi}}\widetilde{\Ric}$.
Since $\tilde{\xi}|_\caG=0$, $\widetilde{\Ric}=O(r^{n/2-1})$, and $T\contraction\widetilde{\Ric}=O^-(r^{n/2})$,
$\caL_{\tilde{\xi}}\widetilde{\Ric}$ itself is $O(r^{n/2-1})$ and
$T\contraction\caL_{\tilde{\xi}}\widetilde{\Ric}=O^-(r^{n/2})$.
Therefore $\Ric'_{\tilde{g}}\tilde{\sigma}=O(r^{n/2-1})$,
$T\contraction\Ric'_{\tilde{g}}\tilde{\sigma}=O^-(r^{n/2})$.
Moreover, $\caB_{\tilde{g}}\tilde{\sigma}=O(r^{n/2})$ and hence
$\divergence_{\tilde{g}}^*\caB_{\tilde{g}}\tilde{\sigma}=O(r^{n/2-1})$,
$T\contraction\divergence_{\tilde{g}}^*\caB_{\tilde{g}}\tilde{\sigma}=O^-(r^{n/2})$.
Thus we conclude
\begin{equation}
	\label{eq:ModificationOfNormalSolution}
	\ambientLichnerowicz\tilde{\sigma}=O(r^{n/2-1}),\qquad
	T\contraction\ambientLichnerowicz\tilde{\sigma}=O^-(r^{n/2}).
\end{equation}

\begin{lem}
	\label{lem:NormalSolutionModifiedToHarmonicSolution}
	Let $\tilde{\sigma}_\mathrm{norm}$ and $\tilde{\xi}$ be as in Lemma \ref{lem:HarminizationOfNormalSolution}.
	Then $\tilde{\sigma}=\tilde{\sigma}_\mathrm{norm}+\caK_{\tilde{g}}\tilde{\xi}\in\tilde{\caS}^X_\mathrm{aTT}(2)$
	and it is a solution to \eqref{eq:ModificationOfNormalSolution}.
\end{lem}

\begin{proof}
	It remains to show that $\tilde{\sigma}\in\tilde{\caS}^X_\mathrm{aTT}(2)$.
	By taking the trace of \eqref{eq:ModificationOfNormalSolution}, we obtain
	$\tilde{\Delta}(\tr_{\tilde{g}}\tilde{\sigma})=O(r^{n/2-1})$.
	In addition, since $\tr_{\tilde{g}}\caK_{\tilde{g}}\tilde{\xi}=O(r)$,
	we have $(\tr_{\tilde{g}}\tilde{\sigma})|_\caG=0$.
	Hence, by Lemma \ref{lem:SolvingLaplacian}, $\tr_{\tilde{g}}\tilde{\sigma}=O(r^{n/2})$.
	Then $\caB_{\tilde{g}}\tilde{\sigma}=O(r^{n/2})$ implies $\divergence_{\tilde{g}}\tilde{\sigma}=O^-(r^{n/2})$.
	Furthermore,
	\begin{equation*}
		\tilde{\Delta}(\tensor{T}{^J}\tensor{\tilde{\sigma}}{_I_J})
		=\tensor{T}{^J}\tilde{\Delta}\tensor{\tilde{\sigma}}{_I_J}
		-2\tensor{\tilde{\nabla}}{^J}\tensor{\tilde{\sigma}}{_I_J}
		=\tensor{T}{^J}\ambientLichnerowicz\tensor{\tilde{\sigma}}{_I_J}
		-2\tensor{\tilde{\nabla}}{^J}\tensor{\tilde{\sigma}}{_I_J}
		=O^-(r^{n/2})
	\end{equation*}
	and
	\begin{equation*}
		\tensor{T}{^J}\tensor{\tilde{\sigma}}{_I_J}
		=\tensor{T}{^J}\tensor{(\caK_{\tilde{g}}\tilde{\xi})}{_I_J}
		=\tensor{T}{^J}\tensor{\tilde{\nabla}}{_I}\tensor{\tilde{\xi}}{_J}
		+\tensor{T}{^J}\tensor{\tilde{\nabla}}{_J}\tensor{\tilde{\xi}}{_I}
		=\tensor{\tilde{\nabla}}{_I}(\tensor{T}{^J}\tensor{\tilde{\xi}}{_J})=O(r).
	\end{equation*}
	Since $\tilde{\Delta}(r^{n/2}\tilde{f}\tensor{T}{_I})=-2nr^{n/2-1}\tilde{f}\tensor{T}{_I}+O(r^{n/2})$
	for $\tilde{f}\in\tilde{\caE}(-n)$,
	one can determine $\tilde{f}$ so that
	$\tilde{\Delta}(\tensor{T}{^J}\tensor{\tilde{\sigma}}{_I_J}+r^{n/2}\tilde{f}\tensor{T}{_I})=O(r^{n/2})$.
	Then $\tensor{T}{^J}\tensor{\tilde{\sigma}}{_I_J}+r^{n/2}\tilde{f}\tensor{T}{_I}$ is still $O(r)$,
	and hence $T\contraction\tilde{\sigma}=O^-(r^{n/2+1})$ by Lemma \ref{lem:SolvingLaplacian}.
\end{proof}

\begin{lem}
	\label{lem:LichnerowiczEncodesObstruction}
	Let $\tilde{\sigma}_\mathrm{norm}$ and $\tilde{\xi}$ be as in Lemma \ref{lem:HarminizationOfNormalSolution}
	and set $\tilde{\sigma}=\tilde{\sigma}_\mathrm{norm}+\caK_{\tilde{g}}\tilde{\xi}$.
	Then $\ambientLichnerowicz\tilde{\sigma}-2\Ric'_{\tilde{g}}\tilde{\sigma}_\mathrm{norm}=O(r^{n/2-1})$, and
	$(r^{1-n/2}(\ambientLichnerowicz\tilde{\sigma}-2\Ric'_{\tilde{g}}\tilde{\sigma}_\mathrm{norm}))|_{T\caG}$
	vanishes.
\end{lem}

\begin{proof}
	Recall that
	\begin{equation*}
		\frac{1}{2}\ambientLichnerowicz\tilde{\sigma}-\Ric'_{\tilde{g}}\tilde{\sigma}_\mathrm{norm}
		=\Ric'_{\tilde{g}}\caK_{\tilde{g}}\tilde{\xi}-\divergence_{\tilde{g}}^*\caB_{\tilde{g}}\tilde{\sigma}
		=\caL_{\tilde{\xi}}\widetilde{\Ric}-\divergence_{\tilde{g}}^*\caB_{\tilde{g}}\tilde{\sigma}.
	\end{equation*}
	Let $\widetilde{\Ric}=r^{n/2-1}S$ and $\tilde{\xi}=rV$.
	We proved in Lemma \ref{lem:HarminizationOfNormalSolution} that $\tensor{T}{^I}\tensor{V}{_I}=O(r)$.
	As in the proof of Lemma \ref{lem:ActionOfNablaT}, we compute
	\begin{equation*}
		\tensor{(\caL_{\tilde{\xi}}\widetilde{\Ric})}{_I_J}
		=\tensor{\tilde{\xi}}{^K}\tensor{\tilde{\nabla}}{_K}\tensor{\widetilde{\Ric}}{_I_J}
		+2\tensor{\widetilde{\Ric}}{_K_(_I}\tensor{\tilde{\nabla}}{_J_)}\tensor{\tilde{\xi}}{^K}
		=4r^{n/2-1}\tensor{S}{_K_(_I}\tensor{T}{_J_)}\tensor{V}{^K}+O(r^{n/2}).
	\end{equation*}
	Thus $(r^{1-n/2}\caL_{\tilde{\xi}}\widetilde{\Ric})|_{T\caG}$ vanishes.
	On the other hand, if we write $\caB_{\tilde{g}}\tilde{\sigma}=r^{n/2}\tilde{\tau}$, then
	\begin{equation*}
		\tensor{(\divergence_{\tilde{g}}^*\caB_{\tilde{g}}\tilde{\sigma})}{_I}
		=\tensor{\tilde{\nabla}}{_(_I}\tensor{(r^{n/2}\tilde{\tau})}{_J_)}
		=nr^{n/2-1}\tensor{T}{_(_I}\tensor{\tilde{\tau}}{_J_)}+O(r^{n/2}),
	\end{equation*}
	and hence
	$(r^{1-n/2}\divergence_{\tilde{g}}^*\caB_{\tilde{g}}\tilde{\sigma})|_{T\caG}=0$.
	This completes the proof.
\end{proof}

\begin{thm}
	\label{thm:DifferentialOfObstruction}
	Let $(M,[g])$ be a conformal manifold of even dimension $n$.
	Then the differential of the obstruction tensor $\caO'_{\bm{g}}$ is given by
	\begin{equation}
		\label{eq:DifferentialObstruction}
		\caO'_{\bm{g}}\varphi=\frac{(-1)^{n/2-1}}{2(n-2)}P\varphi
		+\frac{1}{n+2}\braket{\caO,\varphi}_{\bm{g}}\bm{g}.
	\end{equation}
\end{thm}

\begin{proof}
	Let $\tilde{\sigma}_\mathrm{norm}$, $\tilde{\xi}$ as in Lemma \ref{lem:HarminizationOfNormalSolution}
	and $\tilde{\sigma}=\tilde{\sigma}_\mathrm{norm}+\caK_{\tilde{g}}\tilde{\xi}$.
	By Lemma \ref{lem:NormalSolutionModifiedToHarmonicSolution},
	$P\varphi$ is equal to the trace-free part of
	$2^{n-2}(n/2-1)!^2(r^{1-n/2}\ambientLichnerowicz\tilde{\sigma})|_{TM}$.
	By the previous lemma,
	$(r^{1-n/2}\ambientLichnerowicz\tilde{\sigma})|_{TM}=
	(2r^{1-n/2}\Ric'_{\tilde{g}}\tilde{\sigma}_\mathrm{norm})|_{TM}=c_n^{-1}\caO'_{\bm{g}}\varphi$.
	Therefore,
	\begin{equation*}
		\tf_{\bm{g}}\caO'_{\bm{g}}\varphi=\frac{(-1)^{n/2-1}}{2(n-2)}P\varphi.
	\end{equation*}
	On the other hand, $\tr_{\bm{g}}\caO'_{\bm{g}}\varphi=\braket{\caO,\varphi}_{\bm{g}}$,
	for $\tr_{\bm{g}}\caO=0$ for any $\bm{g}$.
	Hence \eqref{eq:DifferentialObstruction}.
\end{proof}

Combining the theorem above and equation \eqref{eq:SecondDerivativeOfTotalQAndDifferentialObstruction},
we obtain the following.

\begin{cor}
	Let $(M,[g])$ be a compact conformal manifold of even dimension $n\ge 4$ with vanishing obstruction tensor.
	Let $\bm{g}_t$ be a family of conformal structures such that $\bm{g}_0=\bm{g}$.
	Then the second derivative of the total $Q$-curvature at $t=0$ is
	\begin{equation*}
		\left.\frac{d^2}{dt^2}\overline{Q}_t\right|_{t=0}=-\frac{1}{4}\int_M\braket{P\varphi,\varphi}_{\bm{g}},
	\end{equation*}
	where $\varphi=\dot{\bm{g}}_t|_{t=0}$ and
	$P\colon\caS_0(2)\to\caS_0(2-n)$ is the critical GJMS operator on trace-free symmetric 2-tensors.
\end{cor}

\section{Explicit calculations for conformally Einstein manifolds}
\label{sec:ConfEinstein}

Recall that, for $g\in[g]$ Einstein with $\tensor{\Ric}{_i_j}=2\lambda(n-1)\tensor{g}{_i_j}$
so that $\tensor{P}{_i_j}=\lambda\tensor{g}{_i_j}$,
the following formula gives an ambient metric that is Ricci-flat:
\begin{equation}
	\label{eq:AmbientForConformallyEinstein}
	\tilde{g}=2\rho\, dt^2+2t\,dt\,d\rho+t^2(1+\lambda\rho)^2g.
\end{equation}
The inverse of $\tilde{g}$ is
\begin{equation*}
	\tensor{(\tilde{g}^{-1})}{^I^J}=
	\begin{pmatrix}
		0 & 0 & t^{-1}\\
		0 & t^{-2}(1+\lambda\rho)^{-2}\tensor{g}{^i^j} & 0\\
		t^{-1} & 0 & -2t^{-2}\rho
	\end{pmatrix}
\end{equation*}
and the Christoffel symbol of $\tilde{g}$ is given by
\allowdisplaybreaks%
\begin{align*}
	\tensor{\tilde{\Gamma}}{^0_I_J}&=
	\begin{pmatrix}
		0 & 0 & 0 \\ 0 & -\lambda t(1+\lambda\rho)\tensor{g}{_i_j} & 0 \\ 0 & 0 & 0
	\end{pmatrix},\\
	\tensor{\tilde{\Gamma}}{^k_I_J}&=
	\begin{pmatrix}
		0 & t^{-1}\tensor{\delta}{_j^k} & 0 \\
		t^{-1}\tensor{\delta}{_i^k} &
		\tensor{\Gamma}{^k_i_j} &
		\lambda(1+\lambda\rho)^{-1}\tensor{\delta}{_i^k} \\
		0 & \lambda(1+\lambda\rho)^{-1}\tensor{\delta}{_j^k} & 0
	\end{pmatrix},\\
	\tensor{\tilde{\Gamma}}{^\infty_I_J}&=
	\begin{pmatrix}
		0 & 0 & t^{-1} \\ 0 & -(1+\lambda\rho)(1-\lambda\rho)\tensor{g}{_i_j} & 0 \\ t^{-1} & 0 & 0
	\end{pmatrix}.
\end{align*}
\allowdisplaybreaks[0]%
A direct computation shows that $\tensor{\tilde{W}}{_i_j_k_l}=t^2\tensor{W}{_i_j_k_l}$,
where $\tilde{W}$ and $W$ are the Weyl tensors of $\tilde{g}$ and $g$, respectively
(the latter is trivially extended to $\tilde{\caG}=\bbR_+\times M\times\bbR$).
The other components of $\tilde{W}$ are zero.

\begin{lem}
	\label{lem:AmbientLichnerowiczOfConfEinstein}
	Let $\tilde{g}$ as above,
	and suppose that $\tilde\sigma\in\tilde{\caS}(w)$ is of the form
	\begin{equation*}
		\tensor{\tilde{\sigma}}{_i_j}=t^{w}(1+\lambda\rho)^{w}\tensor{\sigma}{_i_j},
	\end{equation*}
	where $\tensor{\sigma}{_i_j}$ is a symmetric 2-tensor on $(M,g)$.
	Then
	\begin{equation}
		\label{eq:AmbLichnerowiczConfEinstein}
			\ambientLichnerowicz\tilde{\sigma}
			=t^{w-2}(1+\lambda\rho)^{w-2}
			(\Delta_\mathrm{L}-4(n-1)\lambda-2(w-2)(n+w-3)\lambda)\sigma,
	\end{equation}
	where $\Delta_\mathrm{L}=\Delta+4n\lambda-2\mathring{W}$ is the Lichnerowicz Laplacian of $g$.
\end{lem}

\begin{proof}
	\allowdisplaybreaks%
	The first covariant derivative of $\tilde{\sigma}$ is as follows:
	\begin{align*}
		\tensor{\tilde{\nabla}}{_\infty}\tensor{\tilde{\sigma}}{_i_j}
		&=\partial_\rho\tensor{\tilde{\sigma}}{_i_j}
		-2\tensor{\tilde{\Gamma}}{^k_\infty_(_i}\tensor{\tilde{\sigma}}{_j_)_k}
		=t^{w}(1+\lambda\rho)^{w-1}(w-2)\lambda\tensor{\sigma}{_i_j},\\
		\tensor{\tilde{\nabla}}{_0}\tensor{\tilde{\sigma}}{_i_j}
		&=\partial_t\tensor{\tilde{\sigma}}{_i_j}
		-2\tensor{\tilde{\Gamma}}{^k_0_(_i}\tensor{\tilde{\sigma}}{_j_)_k}
		=t^{w-1}(1+\lambda\rho)^{w}(w-2)\tensor{\sigma}{_i_j},\\
		\tensor{\tilde{\nabla}}{_k}\tensor{\tilde{\sigma}}{_i_j}
		&=\partial_{x^k}\tensor{\tilde{\sigma}}{_i_j}
		-2\tensor{\tilde{\Gamma}}{^l_k_(_i}\tensor{\tilde{\sigma}}{_j_)_l}
		=t^{w}(1+\lambda\rho)^{w}\tensor{\nabla}{_k}\tensor{\sigma}{_i_j},\\
		\tensor{\tilde{\nabla}}{_k}\tensor{\tilde{\sigma}}{_i_\infty}
		&=-\tensor{\tilde{\Gamma}}{^l_k_\infty}\tensor{\tilde{\sigma}}{_i_l}
		=-t^{w}(1+\lambda\rho)^{w-1}\lambda\tensor{\sigma}{_i_k},\\
		\tensor{\tilde{\nabla}}{_k}\tensor{\tilde{\sigma}}{_i_0}
		&=-\tensor{\tilde{\Gamma}}{^l_k_0}\tensor{\tilde{\sigma}}{_i_l}
		=-t^{w-1}(1+\lambda\rho)^{w}\tensor{\sigma}{_i_k}.
	\end{align*}
	Therefore,
	\begin{align*}
		\begin{split}
			\tensor{\tilde{\nabla}}{_0}\tensor{\tilde{\nabla}}{_\infty}\tensor{\tilde{\sigma}}{_i_j}
			&=\partial_t\tensor{\tilde{\nabla}}{_\infty}\tensor{\tilde{\sigma}}{_i_j}
			-\tensor{\tilde{\Gamma}}{^\infty_0_\infty}\tensor{\tilde{\nabla}}{_\infty}\tensor{\tilde{\sigma}}{_i_j}
			-2\tensor{\tilde{\Gamma}}{^k_0_(_i_|}
			\tensor{\tilde{\nabla}}{_\infty}\tensor{\tilde{\sigma}}{_|_j_)_k}\\
			&=t^{w-1}(1+\lambda\rho)^{w-1}(w-2)(w-3)\lambda\tensor{\sigma}{_i_j},
		\end{split}\\
		\tensor{\tilde{\nabla}}{_\infty}\tensor{\tilde{\nabla}}{_0}\tensor{\tilde{\sigma}}{_i_j}
		&=\tensor{\tilde{\nabla}}{_0}\tensor{\tilde{\nabla}}{_\infty}\tensor{\tilde{\sigma}}{_i_j}
		-2\tensor{\tilde{R}}{_\infty_0^k_(_i}\tensor{\tilde{\sigma}}{_j_)_k}
		=\tensor{\tilde{\nabla}}{_0}\tensor{\tilde{\nabla}}{_\infty}\tensor{\tilde{\sigma}}{_i_j},\\
		\tensor{\tilde{\nabla}}{_\infty}\tensor{\tilde{\nabla}}{_\infty}\tensor{\tilde{\sigma}}{_i_j}
		&=\partial_\rho\tensor{\tilde{\nabla}}{_\infty}\tensor{\tilde{\sigma}}{_i_j}
		-2\tensor{\tilde{\Gamma}}{^k_\infty_(_i_|}
		\tensor{\tilde{\nabla}}{_\infty}\tensor{\tilde{\sigma}}{_|_j_)_k}
		=t^{w}(1+\lambda\rho)^{w-2}(w-2)(w-3)\lambda^2\tensor{\sigma}{_i_j},\\
		\begin{split}
			\tensor{g}{^k^l}\tensor{\tilde{\nabla}}{_k}\tensor{\tilde{\nabla}}{_l}\tensor{\tilde{\sigma}}{_i_j}
			&=\partial_{x^k}\tensor{\tilde{\nabla}}{_l}\tensor{\tilde{\sigma}}{_i_j}
			-\tensor{\tilde{\Gamma}}{^m_k_l}\tensor{\tilde{\nabla}}{_m}\tensor{\tilde{\sigma}}{_i_j}
			-2\tensor{\tilde{\Gamma}}{^m_k_(_i_|}\tensor{\tilde{\nabla}}{_l}\tensor{\tilde{\sigma}}{_|_j_)_m}\\
			&\phantom{=\;}
			-\tensor{\tilde{\Gamma}}{^\infty_k_l}\tensor{\tilde{\nabla}}{_\infty}\tensor{\tilde{\sigma}}{_i_j}
			-2\tensor{\tilde{\Gamma}}{^\infty_k_(_i_|}
			\tensor{\tilde{\nabla}}{_l}\tensor{\tilde{\sigma}}{_|_j_)_\infty}
			-\tensor{\tilde{\Gamma}}{^0_k_l}\tensor{\tilde{\nabla}}{_0}\tensor{\tilde{\sigma}}{_i_j}
			-2\tensor{\tilde{\Gamma}}{^0_k_(_i_|}
			\tensor{\tilde{\nabla}}{_l}\tensor{\tilde{\sigma}}{_|_j_)_0}\\
			&=-t^{w}(1+\lambda\rho)^{w}
			(\Delta\tensor{\sigma}{_i_j}-2(n(w-2)-2)\lambda\tensor{\sigma}{_i_j})
		\end{split}
	\end{align*}
	and hence
	\begin{equation*}
		\begin{split}
			\tilde{\Delta}\tensor{\tilde{\sigma}}{_i_j}
			&=-2t^{-1}\tensor{\tilde{\nabla}}{_0}\tensor{\tilde{\nabla}}{_\infty}\tensor{\tilde{\sigma}}{_i_j}
			+2t^{-2}\rho
			\tensor{\tilde{\nabla}}{_\infty}\tensor{\tilde{\nabla}}{_\infty}\tensor{\tilde{\sigma}}{_i_j}
			-t^{-2}(1+\lambda\rho)^{-2}
			\tensor{g}{^k^l}\tensor{\tilde{\nabla}}{_k}\tensor{\tilde{\nabla}}{_l}\tensor{\tilde{\sigma}}{_i_j}\\
			&=t^{w-2}(1+\lambda\rho)^{w-2}
			(\Delta+4\lambda-2(w-2)(n+w-3)\lambda)\tensor{\sigma}{_i_j}.
		\end{split}
	\end{equation*}
	Consequently, $\ambientLichnerowicz\tilde{\sigma}=(\tilde{\Delta}-2\mathring{\tilde{W}})\tilde{\sigma}$ is
	given by \eqref{eq:AmbLichnerowiczConfEinstein}.
\end{proof}

\begin{thm}
	\label{thm:CriticalGJMSForConfEinstein}
	Let $(M,[g])$ be a conformally Einstein manifold with $\dim M=n\ge 3$,
	and $g\in[g]$ an Einstein representative with Schouten tensor $\tensor{P}{_i_j}=\lambda\tensor{g}{_i_j}$.
	Then, the action of $P_k$
	restricted to $\caS^g_\mathrm{TT}(-n/2+2+k)$ is given by \eqref{eq:Intro:CharacteristicGJMSOperatorOnTTTensor}.
\end{thm}

\begin{proof}
	Let $\varphi=t^{-n/2+2+k}\overline{\varphi}\in\caS^g_\mathrm{TT}(-n/2+2+k)$ and
	$\tilde{\sigma}=(1+\lambda\rho)^{-n/2+2+k}\varphi$. Then
	\begin{equation*}
		\tensor{\tilde{\nabla}}{_k}\tensor{\tilde{\sigma}}{_i_j}
		=t^{-n/2+2+k}(1+\lambda\rho)^{-n/2+2+k}\tensor{\nabla}{_k}\tensor{\overline{\varphi}}{_i_j},\qquad
		\tensor{\tilde{\nabla}}{_\infty}\tensor{\tilde{\sigma}}{_0_i}
		=\tensor{\tilde{\nabla}}{_0}\tensor{\tilde{\sigma}}{_\infty_i}
		=\tensor{\tilde{\nabla}}{_\infty}\tensor{\tilde{\sigma}}{_\infty_i}
		=0.
	\end{equation*}
	Since $\overline{\varphi}$ is a TT-tensor on $(M,g)$, $\tilde{\sigma}$ itself is a TT-tensor with respect to
	$\tilde{g}$, and hence is an extension of the ambient lift of $\varphi$.
	We may compute $\ambientLichnerowicz^k\tilde{\sigma}$ by Lemma \ref{lem:AmbientLichnerowiczOfConfEinstein}.
	By taking the value along $\caG$ and trivializing with respect to $g$,
	we obtain \eqref{eq:Intro:CharacteristicGJMSOperatorOnTTTensor}.
\end{proof}

Now we prove our main theorem.

\begin{proof}[Proof of Theorem \ref{thm:QCurvLocalMaximum}]
	Let $\varphi=\caK_{[g]}\xi+\varphi^g_\mathrm{TT}$ be the decomposition
	of $\varphi=\dot{\bm{g}}_t|_{t=0}$ with respect to \eqref{eq:TTDecomposition}
	and $\Xi_t$ the flow generated by $\xi$.
	Then $\bm{g}'_t=\Xi_{-t}^*\bm{g}_t$ satisfies $\dot{\bm{g}}'_t|_{t=0}=\varphi^g_\mathrm{TT}$
	and the total $Q$-curvature of $\bm{g}'_t$ is equal to $\overline{Q}_t$.
	Therefore
	\begin{equation*}
		\left.\frac{d^2}{dt^2}\overline{Q}_t\right|_{t=0}
		=\left.\frac{d^2}{dt^2}\overline{Q}'_t\right|_{t=0}
		=-\frac{1}{4}\int_M\braket{P\varphi^g_\mathrm{TT},\varphi^g_\mathrm{TT}},
	\end{equation*}
	and thus \eqref{eq:SecondDerivativeOfTotalQForConfEinstein} follows from
	Theorem \ref{thm:CriticalGJMSForConfEinstein}.
	Under the assumption of the latter half of the theorem, any eigenvalue of
	$\Delta_\mathrm{L}|_{\caS_\mathrm{TT}^g(2)}-4(n-1)\lambda+4m(n-2m-1)\lambda$
	is strictly positive for $0\le m\le n/2-1$.
	Therefore, if $\varphi^g_\mathrm{TT}\not=0$,
	the second derivative of $\overline{Q}_t$ at $t=0$ is negative.
\end{proof}

\end{document}